\documentclass[12pt]{article}
\usepackage[utf8]{inputenc}

\usepackage{amsmath,amsfonts,amsthm,geometry,graphicx,float}
\usepackage{pxfonts}
\usepackage[center]{titlesec}
\usepackage{lipsum}

\titleformat*{\section}{\fontsize{14}{20}\selectfont}
\titleformat*{\subsection}{\fontsize{13}{17}\selectfont}

\theoremstyle{definition}

\newcommand{\R}{\mathbb{R}}

\newcommand{\Z}{\mathbb{Z}}
\newcommand{\N}{\mathbb{N}}

\newcommand{\Er}{\mathbb{E}}
\newcommand{\Ha}{\mathbb{H} \mathbb{A}}

\newtheorem{thm}{Theorem}[section]
\newtheorem{defi}[thm]{Definition}
\newtheorem{lem}[thm]{Lemma}
\newtheorem{cor}[thm]{Corollary}
\newtheorem{prp}[thm]{Proposition}
\newtheorem{exm}[thm]{Example}
\newtheorem{rem}[thm]{Remark}

\newtheorem{qst}[thm]{Question}

\DeclareMathOperator{\inv}{inv}
\DeclareMathOperator{\diam}{diam}

\DeclareMathOperator{\img}{Im}

\title{On The Whisker Topology}
\author{John K. Aceti}
\date{\today}

\begin{document}
\maketitle

\begin{abstract}
The purpose of this paper is to explore properties of the whisker topology, which is a topology endowed on the fundamental group and whose utility is to detect locally complicated phenomena in pathological topological spaces. We show that the whisker topology preserves products, resolve an open question regarding the existence of a space which makes $\pi_1^{wh}(X,x_0)$ non-discrete, non-abelian, and Hausdorff, and show the whisker topology is not separable on the earring group $\pi_1(\Er^1,x_0)$.
\end{abstract}

\begin{centering}
  
\section{\centering Introduction}

\end{centering}

The notion of topologizing the fundamental group first appears in a 1935 paper \cite{hurewicz} by Hurewicz. Later  J. Dugundji \cite{dugundji} expanded on this notion which utilized open covers of a space $X$ to topologize a certain group related to the fundamental group $\pi_1(X,x_0)$. The utility of these results predate the field of ``Shape Theory'' \cite{segal} which extends homotopy theory to spaces which fail to have the type of a $CW$-complex or which fail to have locally ``nice'' topological properties. Research on topological fundamental groups remained dormant for several decades until it regained popularity in a paper by D. Biss \cite{biss}; however, since its publication Biss's paper was retracted for many erroneous statements and incorrect proofs. Despite this, the ``big ideas'' of the paper have inspired many mathematicians to explore the utility of topologized fundamental groups (as well as topologized higher homotopy groups) in the spirit of Biss's ideas. The topology on the fundamental group which Biss considered is known as the ``quotient topology'' (denoted $\pi_1^{qtop}(X,x_0)$) and is often seen to be the most ``natural'' topology to endow on the fundamental group. Other topologies have also been considered various papers such as \cite{brazas},\cite{abdul}.

The whisker topology is a topology on the set $\widetilde{X}$ of path-homotopy classes of paths in a space $X$ starting at a fixed base-point $x_0 \in X$. The end-point projection $p:\widetilde{X} \to X, p([\zeta]) = \zeta(1)$ is continuous and when $X$ meets the usual criteria of covering space theory, $p$ is a universal covering map over $X$ \cite{spanier}. However, outside of the limitations of classical covering space theory, $\widetilde{X}$ often serves as a generalized universal covering space \cite{FZ07} and $p$ retains all of the lifitng properties of standard covering maps. The subspace $p^{-1}(x_0) = \pi_1(X,x_0)$ naturally endows the fundamental group with a functorial topology. The resulting topologized group $\pi_1^{wh}(X,x_0)$ is a left-topological group and has been studied in \cite{pakdam}, \cite{babaee}, \cite{rashid}. The spaces $\pi_1^{wh}(X,x_0)$ and $\widetilde{X}^{wh}$ play an increasingly important role in applications of (generalized) covering space theory, including the higher homotopy theory of Peano continua \cite{engl}.

In this paper, we cover some known properties of the whisker topology and explore some new ones. In  Section 3 we explore these properties and  answer an open question by J. Brazas regarding the existence of a space which makes $\pi_1^{wh}(X,x_0)$ a non-discrete, non-abelian, Hausdorff topological group. In Section 4 we give an analysis of the the connectivity properties of $\pi_1^{wh}(X,x_0)$ and consider what separation axioms hold when the fundamental group is endowed with the whisker topology. Finally, in Section 5 we show that $\pi_1^{wh}(X,x_0)$ can be given a suitable pseudo-metric, a condition for general metrizability, and that $\pi_1^{wh}(\Er_1,x_0)$ fails to be separable.
    
\section{\centering Preliminaries and Notation}

Every topological space $X$ in this paper will be considered path connected, locally path connected and based with some base-point $x_0$. The unit interval $[0,1] \subset \R$ will be denoted by $I$ and will be equipped with its usual inherited subspace topology. Every map $f:X \to Y$ where $X,Y$ are topological spaces is assumed to be continuous. The homomorphism induced by $\pi_1$ by based map $f:(X,x_0) \to (Y,y_0)$ will be denoted by $f_{\#}:\pi_1(X,x_0) \to (Y,y_0)$. A \emph{path} $\zeta:[a,b] \to X$ is a continuous function, in particular if $\zeta(a) = \zeta(b) = x_0$ for $x_0 \in X$ then $\zeta$ is said to be a \emph{loop}. If $[c,d] \subseteq [a,b]$ and $\zeta:[a,b] \to X$ is a path we call the restriction $\zeta|_{[c,d]}$ a \emph{subpath} of $\zeta$. We let $\widetilde{X}$ denote the set of path homotopy classes of paths $\zeta:I \to X$ with $\zeta(0) = x_0$ and $p:\widetilde{X} \to X, p([\zeta]) = \zeta(1)$ the endpoint projection function. Note that $p^{-1}(x_0) = \pi_1(X,x_0)$ is the fundamental group of $X$ at $x_0$. Given a path, we can consider the reverse path $\zeta^{-}:I \to X$ defined $\zeta^{-}(t) = \zeta^{-}(1-t)$. If $\zeta:[a,b] \to X$ and $\eta:[c,d] \to X$ are paths such that $\zeta = \eta \circ h$ for some increasing homeomorphisms $h:[a,b] \to [c,d]$, then we write $\zeta \equiv \eta$ and note that $\equiv$ is an equivalence relation of the set of paths that is finer than path homotopy. We may consider a sequence of paths $\zeta_1,\zeta_2,\zeta_3,\dots, \zeta_n$ such that $\zeta_k(1) = \zeta_{k+1}(0)$for every $k \in \N$. We then define $\prod_{k=1}^{n} \zeta_k = \zeta_1 \cdot \zeta_2 \cdot \zeta_3 \cdots \zeta_n$ which is a path defined as $\zeta_k$ on $[\frac{k-1}{k},\frac{k}{k+1}]$. A sequence $\zeta_1, \zeta_2, \zeta_3, \dots$ is said to be a \emph{null sequence} if $\zeta_k(1) = \zeta_{k+1}(0)$ for every $k \in \N$ and such that every neighborhood of $x$ contains $\zeta_k(I)$ for all but finitely many $k$. Then the \emph{infinite} concatenation of such a \emph{null sequence} is the path $\prod_{k=1}^{\infty} \zeta_k$ which is defined as $\zeta_k$ on $[\frac{k-1}{k},\frac{k}{k+1}]$ and $(\prod_{k=1}^{\infty} \zeta_k)(1) = x.$\\

  In this paper, we will often make use of the notion of a space being \emph{homotopically Hausdorff} \cite{grcon}. We provide this definition here for the reader's convenience.\\

  \begin{defi}
    A space $X$ is \emph{homotopically Hausdorff} at $x \in X$ if for every $[\zeta] \not  \in \pi_1(X,x)\backslash\{1\}$, there exists an open neighborhood $U$ of $x$ so that if $i:U \to X$ is the inclusion map, then $[\zeta] \notin i_{\#}(\pi_1(U,x))$. We say that $X$ is homotopically Hausdorff if it's \emph{homotopically Hausdorff} at all of it's points.
    
  \end{defi}

\section{ \centering Basic Properties and Examples of $\pi_1^{wh}(X,x_0)$}

To endow the fundamental group $\pi_1(X,x_0)$ with the whisker topology; we will describe a basis for a topology on $\widetilde{X}$.\\

\begin{defi}
  The \emph{Whisker Topology} on $\widetilde{X}$ is the topology generated by the sets

  $$N([\zeta],U) = \{[\zeta \cdot \eta] \in \widetilde{X} \mid \img(\eta) \subseteq U \}.$$

  where $U$ is an open neighborhood of $\zeta(1)$ for the path $\zeta:(I,0) \to (X,x_0)$. The resulting space is denoted $\widetilde{X}^{wh}$.
\end{defi}

\begin{defi}
Let $[\eta] \in \pi_1(X,x_0)$ and $U$ be an open neighborhood of $x_0$. Then the whisker topology is generated by the basic open sets of the form:

$$B([\zeta],U) = N([\zeta],U) \cap \pi_1(X,x_0).$$

\end{defi}

It's well know that the whisker topology is indeed a well-defined topology on $\widetilde{X}$ (see \cite{spanier} for a comparison to other topologies). Note that, $\pi_1^{wh}(X,x_0)$ has the subspace topology inherited from $\widetilde{X}^{wh}$.

\begin{figure}[h!]
  \centering
  \includegraphics[scale=0.5]{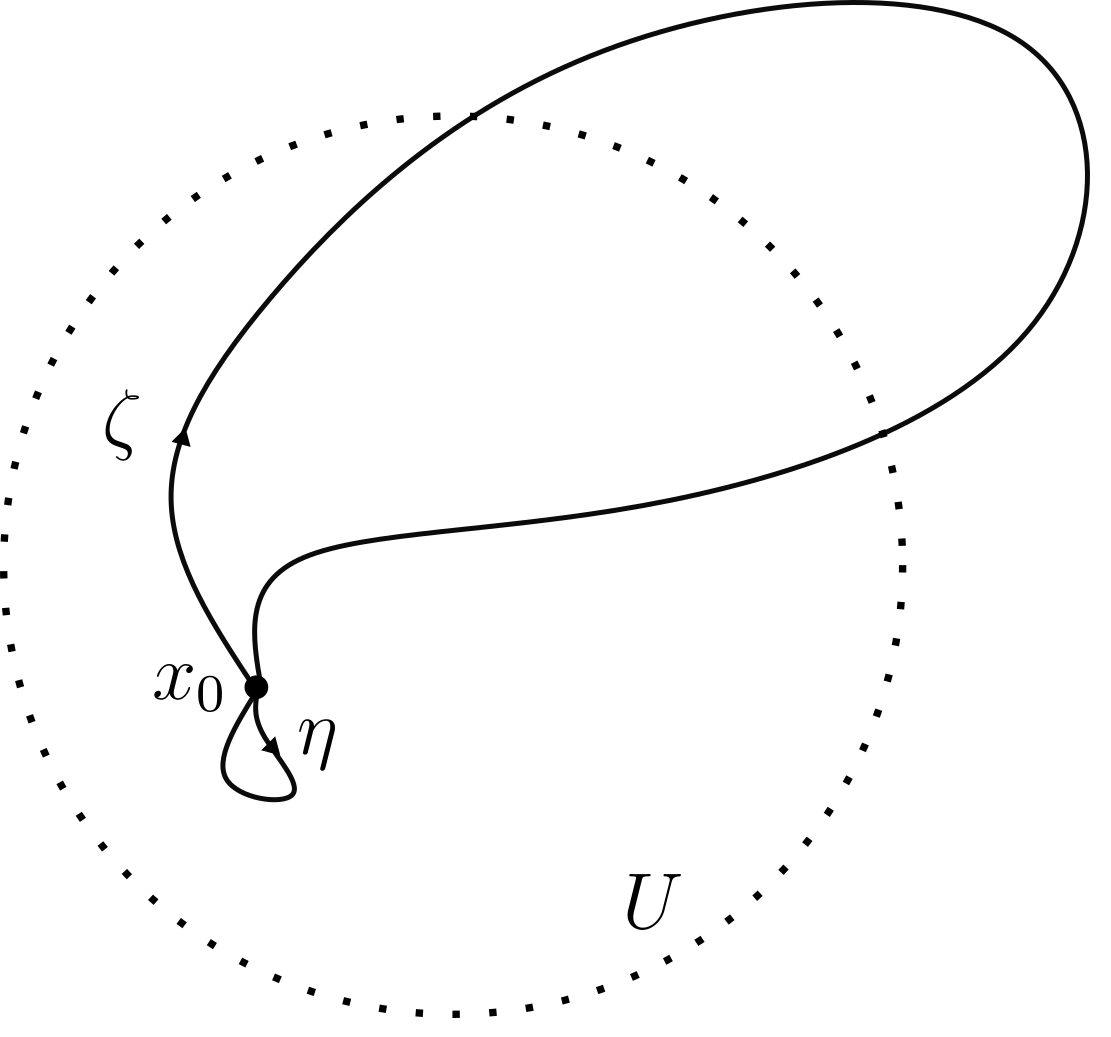}
  \caption{A basic neighborhood $B([\zeta],U)$ of $\pi_1^{wh}(X,x_0)$ consists of homotopy classes $[\zeta \cdot \eta]$ where $\eta$ is a loop in $U$.}

\end{figure}

\begin{defi}
\cite{archangel}: A group $G$ equipped with a topology is a \emph{left-topological group} if for every $g \in G$, left translation $h \mapsto hg$ is a continuous function $G \to G$.

\end{defi}

\begin{prp}
Let $X$ be any space; then, $\pi_1^{wh}(X,x_0)$ is a left topological group.

\end{prp}

\begin{proof}
Let $[\eta] \in \pi_1^{wh}(X,x_0)$ be fixed and consider the left translation $\lambda_{[\eta]}: \pi_1^{wh}(X,x_0) \to \pi_1^{wh}(X,x_0)$ which is defined as $\lambda_{[\eta]}([\kappa]) = [\eta][\kappa]$; then $\lambda_{[\eta]}$ is continuous since $\lambda_{[\eta]}^{-1}(B([\zeta],U)) = B([\eta^{-} \cdot \zeta],U)$. To see $\lambda_{[\eta]}$ is a homeomorphism, let $U$ be any neighborhood of $x_0$ and observe that $\lambda_{[\eta]}(B([\zeta],U) = B([\eta \cdot \zeta],U)$.

\end{proof}

\begin{cor}
  \cite{rashid}: $\pi_1^{wh}(X,x_0)$ is a homogeneous space.
\end{cor}

\begin{prp}
If $f:(X,x_0) \to (Y,y_0)$ is a based map, then $f_{\#}: \pi_1^{wh}(X,x_0) \to \pi_1^{wh}(Y,y_0)$ is continuous.

\end{prp}

\begin{proof}
  Let $B([f \circ \zeta],V)$ be open in $\pi_1^{wh}(Y,y_0)$ where $V$ is open in $Y$ and $[\zeta] \in \pi_1^{wh}(X,x_0)$ with $f_{\#}([\zeta]) = [\eta \cdot \kappa]$ where $\kappa$ has image in $V$. Since $f$ is continuous, we have that $U = f^{-1}(V)$ is open in $X$ and since  $[\zeta] = [\eta \cdot \kappa]$ where $\kappa$ has image in $U$ we see that $f_{\#}([\zeta]) = [(f \circ \eta) \cdot (f \circ \kappa)]$ where $[f \circ \kappa]$ will have image in $V$. This implies that $f_{\#}(B([\zeta],U) \subseteq B([f_{\#}(\zeta)],V)$.

\end{proof}

\begin{cor}
$\pi_1^{wh}$ is a functor from the category of based topological spaces to the category of left topological groups and continuous homeomorphisms.

\end{cor}

\begin{prp}
  Let $\{(X_{i},x_{i}\}_{i \in I}$ be a collection of based spaces. Then the canonical group isomorphism:

  $$\Phi:\pi_1^{wh}\big(\prod_{i \in I} (X_{i},x_{i}) \big) \to \prod_{i \in I} \pi_1^{wh}(X_{i},x_{i})  $$

  is a homeomorphism.
  
\end{prp}

\begin{proof}
  Let $p_i: \prod_{i \in I} X_i \to X_j$ be the $j$-th projection map. Then this induces a continuous homomorphism $(p_j)_{\#}:\pi_1^{wh} \bigg( \prod_{i \in I} (X_i,x_i) \bigg) \to \pi_1^{wh}(X_j,x_j)$. Then together, these maps induce $\Phi$ by $\Phi([\zeta]) = ([p_j \circ \zeta])$. Consequently, the universal property of products ensure that $\Phi$ is continuous. It will suffice to show that $\Phi$ is open. Let $B([\zeta],U)$ be a basic open neighborhood of $[\zeta] = [(\zeta_j)_i]$ in $\pi_1^{wh} \bigg( \prod_{i \in I} (X_i,x_i) \bigg)$. As $\prod_{i \in I} X_i$ is endowed with the product topology we can assume that $U$ has the form:

  $$U = \prod_{i \in F} U_i \times \prod_{i \in I \setminus F} X_i$$

  where $F$ is a finite subset of $I$ and $U_i$ are open neighborhoods of $x_i$ for each $X_i$. Then we claim that:

  $$\Phi(B([\zeta],U)) = \prod_{i \in F} B([\zeta_i],U_i) \times \prod_{i \in I \setminus F} \pi_1^{wh}(X_i,x_i).$$

Indeed, if $[\eta] \in B([\zeta],U)$, write $[\eta] = [\zeta \cdot \kappa], \kappa:I \to U$. In addition, write $\eta = (\eta_i)$ and $\kappa = (\kappa_i)$ with $\img(\kappa_i) \subseteq U_i$ whenever $i \in F$. It follows that $\Phi([\eta]) = ([\zeta_i \cdot \kappa_i]), i \in \N$. Conversely, if $([\eta_i]), i \in \N$, then we have that $([\eta]_i) = [\zeta_i \cdot \kappa_i]$ for loops $\kappa_i:I \to U_i$ when $i \in F$. Now, let $\kappa_i$ be constant at $x_i$ when $i \in F$. Then $(\kappa_i)_{i \in I}:I \to \prod X_i$ has image in $U$ and thus $([\eta_i]) = [(\zeta_i \cdot \kappa_i)] \in B([\zeta],U)$. Hence, we have $([\eta_i]) = \Phi([\eta_i]) \in \Phi(B([\zeta],U))$.

\end{proof}

\begin{thm}
Let $\pi_1^{wh}(X,x_0)$ be the whisker topology on any space $X$. Then the following are equivalent:

\begin{enumerate}
\item Group inversion $\inv: \pi_1^{wh}(X,x_0) \to \pi_1^{wh}(X,x_0)$ is continuous.
\item $\pi_1^{wh}(X,x_0)$ is a topological group.
\item For every $[\eta] \in \pi_1^{wh}(X,x_0)$, conjugation: $c_{[\eta]}:\pi_1^{wh}(X,x_0) \to \pi_1^{wh}(X,x_0), c_{[\eta]}([\kappa]) = [\eta][\kappa][\eta^{-}]$ is continuous.

\end{enumerate}

\begin{proof}
  For $1 \Rightarrow 2$ suppose that group inversion is continuous. Then given $[\zeta],[\eta] \in \pi_1^{wh}(X,x_0)$ it will suffice to show that $([\zeta],[\eta]) \mapsto [\zeta \cdot \eta]$ is continuous. Let $U$ be any neighborhood of $x_0 \in X$ such that $B([\zeta \cdot \eta],U)$ is a basic open neighborhood of $[\zeta \cdot \eta]$. By hypothesis, inversion is continuous and so is continuous at $[\eta^{-}]$. Then we can find an open neighborhood $V$ of $x_0$ such that $V \subseteq U$ and $\inv(B([\eta^{-},V)) \subseteq B([\eta],U)$. We will now demonstrate that group multiplication will map $B([\zeta],V) \times B([\eta],V)$ into $B([\zeta \cdot \eta],U)$. Let $[\zeta \cdot \kappa] \in B([\zeta],V), [\eta \cdot \mu] \in B([\eta],V)$ for loops $\kappa, \mu$ in $V$. Now, $\img (k^{-}) \subseteq V$ and so $[\kappa][\eta] = [\eta][\iota]$ for $\img(\iota) \subseteq U$. Then we have that:

  $$[\zeta \cdot \kappa][\eta \cdot \mu] = [\zeta][\kappa][\eta][\mu] = [\zeta][\eta][\iota][\mu] = [\zeta \cdot \eta][\iota \cdot \mu]$$

  where $\img(\iota \cdot \mu) \subseteq U$. Then this implies that $[\zeta \cdot \kappa][\eta \cdot \mu]$ lies in $B([\zeta \cdot \eta],U))$.

  2. $\Rightarrow$ 3. is an immediate consequence of topological group theory; it suffices to show that (3.) $\Rightarrow$ (1.). Indeed, suppose conjugation is continuous in $\pi_1^{wh}(X,x_0)$. Let $[\eta] \in \pi_1^{wh}(X,x_0)$, it will suffice to show that inversion is continuous at $[\eta]$. Let $U$ be an open neighborhood of $x_0$ in $X$ so that $B([\eta^{-}],U)$ is an open neighborhood of $\inv([\eta]) = [\eta^{-}]$. By hypothesis, conjugation $c_{[\eta]}$ is continuous at the identity element. Then we can find a neighborhood $V$ of $x_0$ such that $[\eta]B([1,V)[\eta^{-}] \subseteq B(1,U)$. That is, if $\iota$ is an loop in $V$ which is based at $x_0$, then there exists a loop $\kappa$ in $U$ such that: $[\eta][\iota][\eta^{-}] = [\kappa]$. Now, it will suffice to show that $\inv(B([\eta],V)) \subseteq B([\eta^{-}],U)$. Let, $\iota$ be a loop in $V$ based at $x_0$; clearly $\iota^{-}$ is in $V$. Then we can find a loop $\kappa$ in $U$ such that $[\eta][\iota^{-}][\eta^{-}] = [\kappa]$. That is, $\inv([\eta \cdot \iota]) = [\iota^{-} \cdot \eta^{-}] = [\eta^{-}][\kappa] \in B([\eta^{-}],U)$.
\end{proof}

\end{thm}

\begin{cor}
$\pi_1^{wh}(X,x_0)$ is a topological group if and only if for every loop $\zeta$ based at $x_0$ and neighborhood $U$ of $x_0$, there exists a neighborhood $V$ of $x_0$ such that for any loop $\eta$ based at $x_0$  in $V$, the conjugate $\zeta \cdot \eta \cdot \zeta^{-}$ is homotopic to some loop in $U$.

\end{cor}

\begin{figure}[h!]
  \centering
  \includegraphics[scale=0.5]{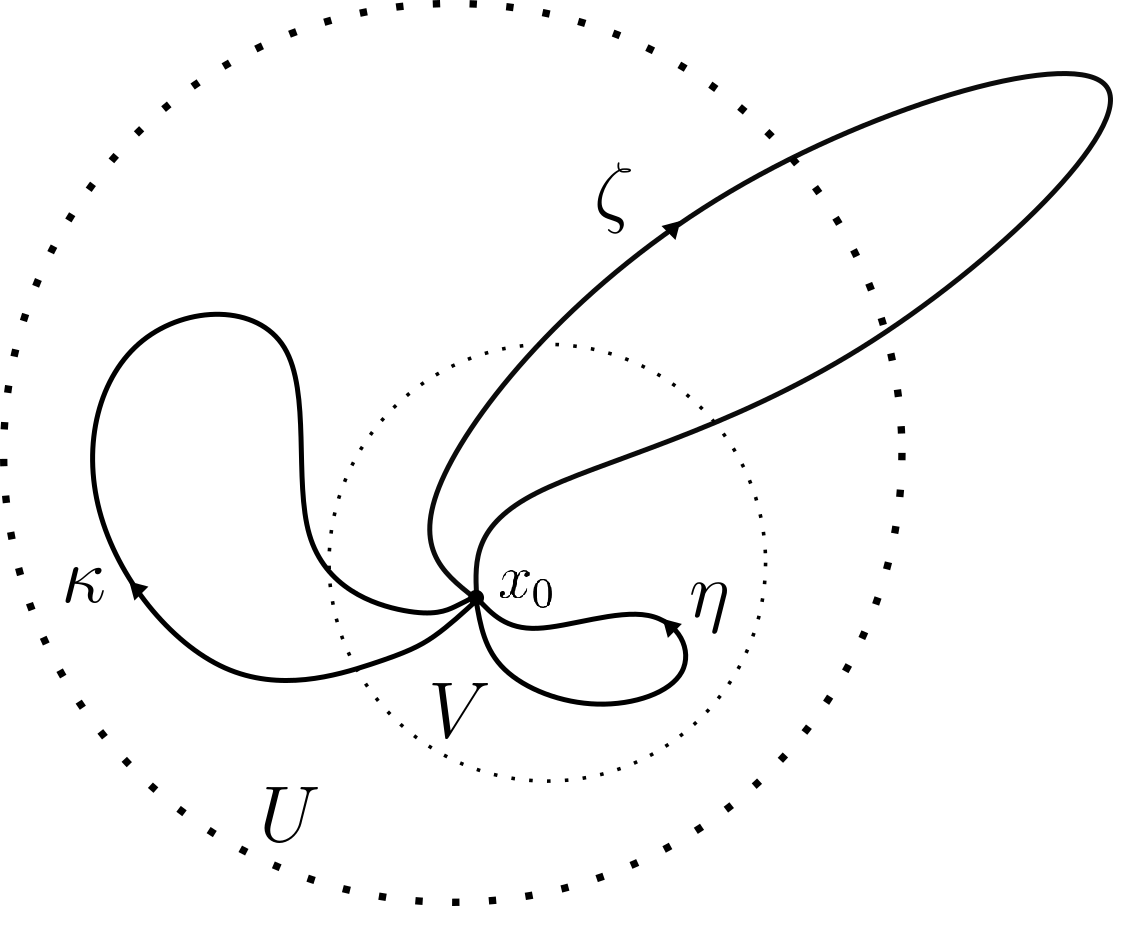}
  \caption{$\pi_1^{wh}(X,x_0)$ is a topological group if and only if for every based loop $\zeta$ and neighborhood $U$ of $x_0$ there exists a neighborhood $V$ of $x_0$ such that for any loop $\eta$ in $V$, the conjugate $\zeta \cdot \eta \cdot \zeta^{-}$ is homotopic to some loop $\kappa$ in $U$.}

\end{figure}

\begin{cor}
  If $\pi_1(X,x_0)$ is abelian, then $\pi_1^{wh}(X,x_0)$ is a topological group.

\end{cor}

\begin{prp}
$\pi_1^{wh}(X,x_0)$ is discrete if and only if $X$ is semilocally-simply connected at $x_0$.

\end{prp}

\begin{proof}
Suppose that $\pi_1^{wh}(X,x_0)$ is discrete. Then there exists an open neighborhood $U$ of $x_0$ such that $B(1,U) = \langle 1 \rangle$. It follows that every loop in $U$ based at $x_0$ is null homotopic in $X$. This implies $X$ is semilocally simly connected at $x_0$. Conversely, suppose that $X$ is semilocally simply connected at $x_0$. Then there exists an open neighborhood $U$ of $x_0$ such that every loop in $U$ based at $x_0$ is nullhomotpic in $X$.
\end{proof}

The converse of Corollary 3.10 in general isn't true; a simple counterexample illustrates this:

\begin{exm}
Let $X = S^1 \vee S^1$ be the wedge sum of two circles (with weak topology). Then it's well known that $\pi_1(X) \cong F_2$ where $F_2$ denote the free group on two generators. This is obviously non-abelian, but then by Proposition 3.8 we have that this topology is discrete and so it's a topological group.

\end{exm}

We will now consider some particular examples of spaces and their relation to the whisker topology $\pi_1^{wh}(X,x_0)$. The most important example of these will be the ``earring space'' $\Er^1$. The earring space $\Er^1$ is the prototypical space which one studies in ``wild topology''.

\begin{prp}
$\pi_1^{wh}(\Er^1,b_0)$ where $b_0=(0,0)$ is the canonical wild-point of $\Er^1$ is not a topological group.

\end{prp}

\begin{proof}
  Let $\Er^1$ be the usual earring space (see Figure 2). Recall that $\Er^1 = \bigcup_{n \geq 1} C_n$ where $C_n \subseteq \R^2$ is the circle of radius $1/n$ which is centered at $(1/n,0)$ for each $n \in \N$. For each $n \geq 1$, let $\ell_n:I \to \Er^1$ be the canonical loops which traverse $C_n$ once counterclockwise which are based at $b_0$. Next, define $U_n$ to be a neighborhood of $b_0$ where $\bigcup_{k \geq n} C_k \subseteq U$ and where $\bigcup_{k \leq n}  C_k$ is an open arc for each $k < n$.
  We proceed by contradiction. Indeed, suppose that $\pi_1^{wh}(\Er^1,b_0)$ is a right-topological group. Consider the sequence $\{[\ell_n \cdot \ell_1] \}_{n \in \N}$; this sequence fails to converge to $[\ell_1]$  as for any basic neighborhood $B([\ell_1],U_n) \; \; (n \geq 2)$ , shape injectivity of the earring space \cite{cannon2} ensures that $[\ell_n \cdot \ell_1] \not \in B([\ell_1],U_n)$ (note that $\{[\ell_n]_{n \in \N} \to 1$). This is a contradiction to our assumption and hence $\pi_1^{wh}(\Er^1,b_0)$ fails to be a topological group.
  
\end{proof}

\begin{rem}
It's important to note that $\pi_1^{wh}(\Er^1,b_0)$ is a topological group whenever $x \neq b_0$. This is a consequence of Proposition 3.12. 

\end{rem}

\begin{figure}[h!]
  \centering
  \includegraphics[scale=1.2]{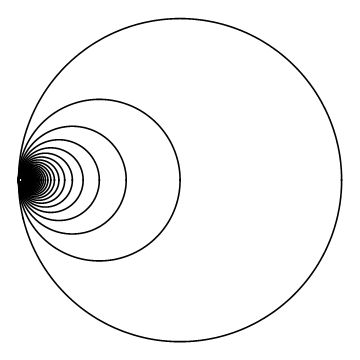}
  \caption{The Earring space $\Er^1$.}

\end{figure}

\begin{cor}
If $X$ is any space where $\Er^1$ is a retract of $X$, then there exists a point  $x_0 \in X$ such that $\pi_1^{wh}(X,x_0)$ is not a topological group. 

\end{cor}

\begin{exm}
Let $\Ha$ be the ``Harmonic Archipelago'' as described in \cite{bogley} (see Figure 4). The space $\Ha$ is a prototypical example of a space which is not homotopically Hausdorff. In \cite{babaee} it's shown that $\pi_1^{wh}(\Ha,x_0)$ is indiscrete (and hence a topological group).
  
\end{exm}

\begin{qst}
Does there exist a space $X$ such that $\pi_1^{wh}(X,x_0)$ is a non-abelian, non-discrete, Hausdorff, topological group?

\end{qst}

J. Brazas asked the previous question on his blog ``Wild Topology'' for which we now given an affirmative answer:


\begin{prp}
Let $X = \prod_{i \in \N} (S^1 \vee S^1)$. Let W = $\pi_1^{wh}(X,x_0)$ where $x_0 = (w_0,w_0,w_0, \dots)$ and $w_0$ is wedge point of $S^1 \vee S^1$. Then $W$ is a non-abelian, Hausdorff, non-discrete topological group.

\end{prp}

\begin{proof}
   By Proposition 3.7 we have that:

  $$\pi_1^{wh} \bigg(\prod_{i \in \N} (S^1 \vee S^1,x_0) \bigg) \cong \prod_{i \in \N} \pi_1^{wh}(S^1 \vee S^1, x_0) \cong (F_2)^{\omega}.$$

  This implies that $W$ is both Hausdorff and non-abelian. It's not too difficult to see that $W$ is not semilocally simply connected at any of its points and hence by Proposition 3.11 we have that $W$ is non-discrete. Finally, since products of topological groups are topological groups \cite{archangel}; we conclude that $W$ is indeed a topological group.
  
\end{proof}

\begin{rem}
  In fact, we can go one step further from Proposition 3.18 above to construct a general class of examples. Let $X = \prod_{i \in \N} \bigg( \bigvee_{i \in I} S^1 \bigg)$ (where the cardinality of $I$ is at most countable) be the usual bouquet of circles. Then we have:

  $$W_{I} = \pi_1^{wh}(X) \cong \prod_{i \in \N} \pi_1^{wh} \bigg( \bigvee_{i \in I} S^1 \bigg).$$

  Then $W_I$ is by similar argument is a Hausdorff, non-abelian, non-discrete topological group.
 
\end{rem}

\begin{figure}[h!]
  \centering
  \includegraphics[scale=0.5]{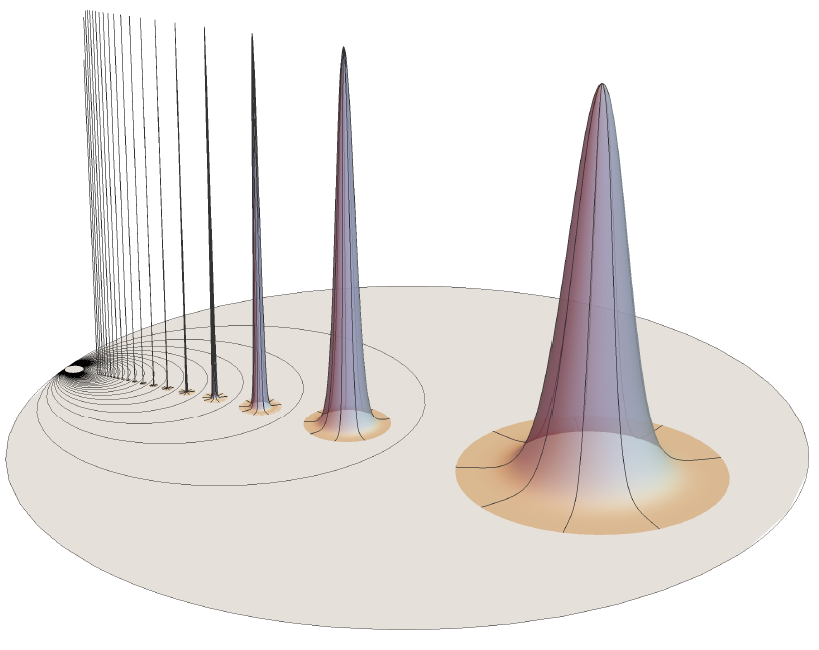}
  \caption{The Harmonic Archipelago $\Ha$.}

\end{figure}

\begin{centering}
  \section{\centering Connectedness \& Separation Axioms of $\pi_1^{wh}(X,x_0)$}
  
\end{centering}

First, we recall an important property of the whisker topology on $\widetilde{X}$.\\

\begin{lem}
Let $U$ be a fixed neighborhood of a particular point $x_0 \in X$. If $\widetilde{X}$ is the set of path homotopy classes starting at $x_0$ then if $[\eta] \in N([\zeta],U)$, then $N([\eta],U) = N([\zeta],U)$.

\end{lem}

\begin{proof}
  First observe that if $[\eta] \in N([\zeta],U)$, then $[\eta] = [\zeta \cdot \kappa], \img(\kappa) \subseteq U$. Then we have that: $[\zeta] = [\eta \cdot \kappa^{-}], \img(\kappa^{-}) \subseteq U$. We will show the inclusion $N([\zeta],U) \subseteq N([\eta],U)$ (the other inclusion is near identical). Indeed, suppose $[\mu] \in N([\zeta],U)$ so that $[\mu] = [\zeta \cdot \iota], \img(\iota) \subseteq U$. Then $[\mu] = [\eta \cdot (\kappa^{-} \cdot \iota)], \img(k^{-} \cdot \iota) \subseteq U$. This implies $[\mu] \in N([\eta],U)$ as desired. 
\end{proof}

\begin{cor}
Let $U$ be any neighborhood of $x_0$ in some space $X$. Then $B([\zeta],U)$ and $B([\eta],U)$ are either equal or disjoint.
\end{cor}

\begin{proof}
Suppose that $B([\zeta],U)$ and $B([\eta],U)$ are not disjoint for loops $\zeta, \eta$ at $x_0$. Then $[\zeta \cdot \kappa] = [\eta \cdot \iota], \img(\kappa) \subseteq U$ and $\img(\iota) \subseteq U$. Then $[\zeta] = [\eta \cdot (\iota \cdot \kappa^{-})], \img(\kappa^{-}) \subseteq U$. Then $[\zeta] \in B([\eta],U)$ which completes the proof.
\end{proof}

Corollary 4.2 implies that all basic open sets of $\pi_1^{wh}(X,x_0)$ are clopen, including subgroups [18, Corollary 3.2]. That is we have the following:

\begin{thm}
If $\pi_1^{wh}(X,x_0)$ is Hausdorff, then $\pi_1^{wh}(X,x_0)$ is zero-dimensional (small inductive dimension).

\end{thm}

\begin{proof}
Any basic neighborhood of the identity $B(1,U)$ where $U$ is an open set of $X$ are clopen. In turn, this implies that if $\mathfrak{N}(x_0)$ is the set of all open neighborhoods of $x_0$ then we have that $\bigcap_{U \in \mathfrak{N}(x_0)} B(1,U)$ is the closure of $\langle 1 \rangle$.
\end{proof}

\begin{cor}
If $\pi_1^{wh}(X,x_0)$ is Hausdorff, then it is totally separated.

\end{cor}


\begin{thm}
  Let $\pi_1^{wh}(X,x_0)$ be the whisker topology. Then the following are equivalent:

  \begin{enumerate}
  \item The trivial subgroup is closed in $\pi_1^{wh}(X,x_0)$.
  \item $\pi_1^{wh}(X,x_0)$ is Hausdorff.
  \item $\pi_1^{wh}(X,x_0)$ is $T_3$.

  \item $\pi_1^{wh}(X,x_0)$ is Hausdorff if and only if $X$ is homotopically Hausdorff at $x_0$.

  \end{enumerate}

\end{thm}

\begin{proof}
  The implication $3 \Rightarrow 2 \Rightarrow 1$ is clear. In addition, it's well known that every zero dimensional Hausdorff space is $T_3$ \cite{archangel}. It will suffice to show $1 \Rightarrow 2$, $2 \Rightarrow 4$ and $4 \Rightarrow 2$.For $1 \Rightarrow 2$: Suppose that the trivial subgroup $\langle 1 \rangle$ is closed and $[\zeta] \neq [\eta]$ in $\pi_1^{wh}(X,x_0)$. Then observe that $[\eta^{-} \cdot \zeta] \neq 1$ and so there exists an open neighborhood $U$ of $x_0$ such that $1 \not \in B([\eta^{-} \cdot \zeta],U)$. To obtain a contradiction, suppose that $B([\zeta],U)$ and $B([\eta],U)$ are disjoint. Then Lemma 4.1 implies that $[\zeta] \in B([\eta],U)$. Now $[\zeta] = [\eta \cdot \iota], \img(\iota) \subseteq U$. Now, $[\eta^{-} \cdot \zeta] = [\iota]$, this implies that $[\eta^{-} \cdot \zeta] \in B(1,U)$. Then $1 \in B([\eta^{-} \cdot \zeta],U)$; a contradiction.

 For $2 \Rightarrow 4$: Suppose $\langle 1 \rangle$ is not closed. Then there exists $[\zeta] \not = 1$ in the closure of $\langle 1 \rangle$. Suppose $U$ is an open neighborhood of $x_0$ so that $B([\zeta],U)$ is an open neighborhood of $[\zeta]$ with $1 \in B([\zeta],U)$. Then we have that $1 = [\zeta][\kappa]$ for some loop $\kappa$ in $U$. Then it follows that $[\zeta] = [\kappa^{-}]$ and thus $X$ is not homotopically-Hausdorff at $x_0$.

  Conversely, suppose $X$ is homotopically Hausdorff at $x_0$. Then there exists a non-null homotopic loop $\zeta$ based at $x_0$, which is path homotopic to a loop in every neighborhood of $x_0$. We will show that $[\zeta]$ lies in the closure of $\langle 1 \rangle$. Indeed, let $U$ be any neighborhood of $x_0$. Then we have $[\zeta] = [\kappa]$ for some loop $\kappa$ in $U$. This implies that $1 = [\zeta \cdot \kappa^{-}]$ since $\kappa^{-}$ has image in $U$, this implies that $1 \in B([\zeta],U)$. Then $1$ lies in any neighborhood of $[\zeta]$ and subsequently, $[\zeta]$ is in the closure of $\langle 1 \rangle$.

\end{proof}

\section{\centering Metrizability, Separability and Compact Subsets of $\pi_1^{wh}(X,x_0)$}

\subsection{Metrizability and Separability of $\pi_1^{wh}(X,x_0)$}

In this section, we will assume that $X$ is a metrizable space and we will let $d$ denote a metric which induces the topology on $X$.

Below, we give a detailed proof that when $X$ is metrizable, $\pi_1^{wh}(X,x_0)$ is pseudometrizable. If $X$ is also homotopically Hausdorff, then $\widetilde{X}$ is Hausdorff and thus mertizable. However, the separability of $\widetilde{X}$ is not entirely clear. There exist Peano continua $X$ for which $\pi_1^{wh}(X,x_0)$ is uncountable and separable. For example if $X = (S^1)^{\omega}$ is the infinite dimensional torus, then $\pi_1^{wh}(X,x_0)$ is isomorphic to the Baer-Specker group. (endowed with the product topology); however, it contains a dense countable subset; namely, $\bigoplus_{\omega} \Z$. At the other extreme, if $X$ is the Griffiths space or harmonic archipelago, then $\pi_1^{wh}(X,x_0)$ is uncountable and indiscrete and hence trivially separable. In considering the infinite earring space $\Er^1$, the separability of $\pi_1^{wh}(\Er^1,b_0)$ and $\widetilde{\Er}^{wh}$ is less clear. While, it's known that $\widetilde{\Er}^1$ has the structure of a topological $R$-tree (metrizable, uniquely arcwise-connected, and locally arcwise connected). The example of the infinite torus makes it clear that uncountability on the fundamental group is not sufficient for non-separability. 

\begin{defi}
  Let $\rho$ be a distance function on $\pi_1^{wh}(X,x_0)$ as follows for $[\zeta],[\eta] \in \pi_1^{wh}(X,x_0)$, set:

  $$\rho([\zeta],[\eta]) = \inf \{\diam(\xi) \mid [\xi] = [\zeta^{-} \cdot \eta] \}. $$

  Here $\diam(\xi)$ denotes the usual diameter of a loop. It is clear that $\rho([\zeta],[\eta]) \geq 0$ and $\rho([\zeta],[\eta]) = \rho([\eta],[\zeta])$. We verify now verify the triangle inequality. 
\end{defi}

\begin{lem}
If $\zeta,\eta$ are loops in $X$ based at $x_0$ and $\xi = \zeta \cdot \eta$, then $\diam(\xi) \leq \diam(\zeta) + \diam(\eta)$.

\end{lem}

\begin{proof}
  If $\diam(\xi) > \diam(\zeta) + \diam(\eta)$, then there exists $x,y \in [0,1]$ where $d(\xi(x),\xi(y)) > \diam(\zeta) + \diam(\eta)$. If $x,y \in [0,1/2]$, then $d(\xi(x),\xi(y)) > \diam(\zeta(x/2),\zeta(y/2)) > \diam(\zeta)$, a contradiction. On the other hand, if $x,y \in [1/2,1]$, then $d(\xi(x),\xi(y)) = d(\eta(2x-1),\zeta(2y-1)) > \diam(\eta)$. Without loss of generality suppose $x \in [0,1/2]$ and $y \in [1/2,1]$. Then we have that:

  $$d(\zeta(2x),\eta(2y-1)) = d(\xi(x),\xi(y)) > \diam(\zeta) + \diam(\eta) \geq d(\zeta(2x),x_0) + d(x_0,\eta(2y-1)).$$.

  This contradicts the triangle inequality.

\end{proof}

\begin{lem}
  Given $[\zeta],[\eta],[\xi] \in \pi_1^{wh}(X,x_0)$, $\rho([\zeta],[\iota]) \leq \rho([\zeta],[\eta]) + \rho([\eta],[\iota])$.
\end{lem}

\begin{proof}
Suppose $\rho([\zeta],[\xi]) > \rho([\zeta],[\eta]) + \rho([\eta],[\xi])$. Then $\rho([\zeta],[\xi]) - \rho([\eta],[\xi]) > \rho([\zeta],[\eta])$, and there exists a loop $\kappa$ with $[\kappa] = [\zeta^{-} \cdot \eta]$ and $\rho([\zeta],[\xi]) - \rho([\eta],[\xi]) > \diam(\kappa)$. Now, since $\rho([\zeta],[\xi]) - \diam(\kappa) > \rho([\eta],[\xi])$, there exists a loop $\theta$ with $[\theta] = [\eta^{-} \cdot \xi]$ and $\rho([\zeta],[\xi]) - \diam(\kappa) > \diam(\theta)$. We have $\diam(\kappa) + \diam(\theta) \geq \diam(\kappa \cdot \theta)$ from Lemma 5.2 and thus $\rho([\zeta],[\xi]) > \diam(\kappa \cdot \theta)$. Since $[\kappa \cdot \theta] = [\zeta^{-} \cdot \xi]$ this yields a contradiction.

\end{proof}

We conclude that $\rho$ is a pseudometric on $\pi_1^{wh}(X,x_0)$.

\begin{lem}
The pseudo-metric $\rho$ induces the topology of $\pi_1^{wh}(X,x_0)$.

\end{lem}

\begin{proof}
  Suppose $B([\zeta],U)$ is a basic neighborhood of $[\zeta] \in \pi_1^{wh}(X,x_0)$. Find some $r > 0$ such that $O_d(x_0,r) \subseteq U$ where $O_d(x_0,r)$ is the open $r$-ball about $x_0$. Suppose that $[\eta] \in O_{\rho}([\zeta],r)$. Since $\rho([\zeta],[\eta]) < r$, there exists some loop $\xi$ based at $x_0$ with $[\xi] = [\zeta^{-} \cdot \eta]$ and $\diam(\xi) < r$. Now, since $\xi$ has image in $O_d(x_0,r)$, it has image in $U$. This implies that $[\eta] = [\zeta \cdot \kappa] \in B([\zeta],U)$.

  Conversely, suppose that $r > 0$ and consider the neighborhood $O_{\rho}([\zeta],r)$. It will suffice to show that $B([\zeta],O_d(x_0,r/3)) \subseteq O_{\rho}([\zeta],r)$. Indeed, suppose $[\eta] \in B([\zeta],O_d(x_0,r/3))$. Then write $[\eta] = [\zeta \cdot \kappa]$ for a loop $\kappa \in O_d(x_0,r/3)$. Then it must be that $[\kappa] = [\zeta^{-} \cdot \eta]$ and $\diam(\kappa) \leq \frac{2r}{3}$. Thus, $\rho([\zeta],[\eta]) \leq \frac{2r}{3} < r$, showing that $[\eta] \in O_{\rho}([\zeta],r)$.

\end{proof}

\begin{thm}
If $X$ is metrizable, then $\pi_1^{wh}(X,x_0)$ is pseudometrizable.

\end{thm}

\begin{proof}
This follows from Lemma 5.2, 5.3 and 5.4.

\end{proof}

\begin{thm}
  Let $X$ be a path-connected, metrizable, homotopically Hausdorff space and $x_0 \in X$. Then $\pi_1^{wh}(X,x_0)$ is a homogeneous, zero-dimensional, metrizable space.

\end{thm}

\begin{proof}
This follows from Theorem 4.6 and Theorem 5.5.
\end{proof}

\begin{defi}
  A path $\zeta:I \to X$ is said to be \emph{reduced} if either $\zeta$ is constant or if there is no subpath of $\zeta$ that is a null-homotopic loop.

\end{defi}

\begin{lem}
  \cite{cannon}: If $X$ is a one-dimensional Hausdorff space, then every path $\zeta:I \to X$ is path-homotopic to a reduced path $\eta:I \to X$ by a homotopy with image in $\img(\zeta)$. Moreover, if $\zeta$ and $\eta$ are path-homotopic reduced paths, then $\zeta \equiv \eta$.

\end{lem}

\begin{thm}
  The group $\pi_1^{wh}(\Er^1,b_0)$ is not separable.
  
\end{thm}


\begin{proof}
We have that $\pi_1^{wh}(\Er^1,b_0)$ is metrizable by Theorem 5.5 and so it suffices to find a subspace of $\pi_1^{wh}(\Er^1,b_0)$ that fails to be separable. Let $S \subseteq \pi_1^{wh}(\Er^1,b_0)$ be a collection of homotopy classes of reduced loops of the form $(\prod_{j=1}^{\infty} \ell_{n_j}) \cdot \ell_m$ where $1 \leq n_1 < n_2 < n_3 < \cdots $ and $m \in \N$. Now, suppose that $A = \{g_1,g_2,g_3, \dots \}$ is a countably infinite subset of $S$ that's dense in $S$ (with subspace topology which is inherited from $\pi_1^{wh}(\Er^1,b_0)$). Write $g_n = [\zeta_k][\ell_{m_k}]$ where $\zeta_k = \prod_{j = 1}^{\infty} \ell_{n_k,j}$ and $m_k \in \N$. Applying Cantor diagonalization find a sequence of integers $1 \leq p_1 < p_2 < p_3 < \cdots$ such that if $\eta = \prod_{j = 1}^{\infty} \ell_{p_j}$ such that $p_k \not = n_k,k$ for each $k$. Since all of these loops are reduced and homotopy classes have reduced representatives (Lemma 5.8), we have that $[\eta] \not = [\zeta_k]$ for all $k$. Let $g = [\eta][\ell_1]$ and note that $g \in S$. Let $U \subseteq \Er^1$ be an open neighborhood such that $\bigcup_{n \geq 2} C_n \subseteq U$ and $U \cap C_1$ is an open arc. We claim that $B(g,U) \cap S$ does not meet $A$. Suppose otherwise, that is suppose we have that $[\eta \cdot \ell_1 \cdot \kappa] \in A$ for a reduced loop $\kappa:I \to U$. Since $\kappa$ does not contain $\ell_1$ as a subloop and $\eta \cdot \ell_1$ is reduced, $\eta \cdot \ell_1 \cdot \kappa$ is reduced. Since $[\eta \cdot \ell_1 \cdot \kappa] = [\zeta_k \cdot \ell_{n_k}]$ for some $k \in \N$, then Lemma 5.8 implies that  $\eta \cdot \ell_1 \cdot \kappa \equiv \zeta_k \cdot \ell_{n_k}$. However, since $\eta$ and $\zeta_k$  are indexed by well-ordered sets, the only way for this to occur is if $\eta = \zeta_k$. This is a contradiction.
  
\end{proof}

\begin{thm}
  If $X$ is any one-dimensional Peano continua and $x_0 \in X$, then we have the following dichotomy. Either

  \begin{enumerate}

  \item $X$ is semi-locally simply connected and $\widetilde{X}^{wh}$ and $\pi_1^{wh}(X,x_0)$ are separable metric spaces.
  \item or $X$ is not semi-locally simply connected and $\widetilde{X}^{wh}$ and $\pi_1^{wh}(X,x_0)$ are non-separable metric spaces.

 \end{enumerate}

\end{thm}

\begin{proof}
  We assume that $X$ is metrizable and so $\widetilde{X}^{wh}$ and $\pi_1^{wh}(X,x_0)$ are pseudometrizable. Moreover, since all one-dimensional Hausdorff spaces are homotopically Hausdorff, $\widetilde{X}^{wh}$ and $\pi_1^{wh}(X,x_0)$ are both Hausdorff and thus both metrizable. If $X$ is semi-locally simply connected, then $\pi_1^{wh}(X,x_0)$ is finitely generated and discrete and $p:\widetilde{X}^{wh} \to X$ is a universal covering map. Let $P(X,x_0)$ be the space of paths $\zeta:I \to X$ with $\zeta(0) = x_0$ equipped with the compact-open topology. Since $X$ is a compact metric space, $P(X,x_0)$ is a separable metric space \cite{engl}. Now, consider the function $\pi: P(X,x_0) \to \widetilde{X}, \pi(\zeta) = [\zeta]$ that identifies path-homotopy classes. It's known that the quotient topology on $\widetilde{X}$ inherited from $\pi$ gives $\widetilde{X}$ the structure of a universal covering space of $X$, and is thus equivalent to the Whisker topology (by the usual classification of covering spaces) \cite{spanier}. Hence, $\pi:P(X,x_0) \to \widetilde{X}^{wh}$ is a continuous map of a separable space onto $\widetilde{X}^{wh}$. Then it follows that the metrizable space $\widetilde{X}$ and its subspace $\pi_1^{wh}(X,x_0)$ are separable.
  On the other hand, suppose that $X$ is not semi-locally simply connected. Then there exists a map $s: \Er^1 \to X$ and $r:X \to \Er^1$ with $r \circ s \simeq 1_{\Er^1}$ \cite{cannon}. Let $x_0 = s(b_0)$. Then functoriality of the Whisker topology on the fundamental group (Corollary 3.5) gives that the induced continuous homomorphisms $s_{\#}:\pi_1^{wh}(X,x_0) \to \pi_1^{wh}(X,x_0)$ and $r_{\#}:\pi_1^{wh}(X,x_0) \to \pi_1^{wh}(\Er^1,b_0)$ satisfy $r_{\#} \circ S_{\#} = 1_{\pi_1^{wh}(X,x_0)}$ From the topological perspective this implies that $s_{\#}$ is an embedding of spaces. Hence, $s_{\#}$ embeds the non-separable space $\pi_1^{wh}(\Er^1,b_0)$ by Theorem 5.7  into the metrizable space $\pi_1^{wh}(X,x_0)$. Now, since subspaces of separable metric spaces are separable, we conclude that $\pi_1^{wh}(X,x_0)$ is not separable. Moreover, $\pi_1^{wh}(X,x_0)$ is a subspace of the metrizable space $\widetilde{X}^{wh}$ so by the same reasoning, $\widetilde{X}^{wh}$ can't be separable.

\end{proof}

  \section*{\centering Acknowledgments}

The author would like to thank Jeremy Brazas for his valuable input and suggestions in writing this paper and for inspiration from his blog ``wild topology''.



\begin{thebibliography}{9}



\bibitem{abdul}
M. Abdullahi Rashid, N. Jamali, B. Mashayekhy, S.Z. Pashaei, H. Torabi, \emph{On subgroup topologies on the fundamental group}. Hacettepe Journal of Mathematics \& Statistics 49 (2020), no. 3, 935 – 949.

\bibitem{archangel}
A. Arhangegel'skii and M. Tkachenko, \emph{Topological Groups and Related Structures}, Atlantic Press, Amsterdam, 2008.


\bibitem{babaee}

A.Babaee, B. Mashayekhy, H. Mirebrahimi, H. Torabi, \emph{On Topological Homotopy Groups and Relation to Hawaiian Groups},

\bibitem{biss}
D. Biss, \emph{The topological fundamental group and generalized covering spaces},
Topology and its Applications 124 (2002), 355–371. RETRACTED.

\bibitem{bogley}
W. Bogley, A. Sieradski, \emph{Universal path spaces}, Unpublished manuscript.

\bibitem{brazas}
  J. Brazas, \emph{The fundamental group as a topological group}, Topology Appl. 160 (2013) 170-188.


\bibitem{cannon}
  J.W. Cannon, G.R. Conner, \emph{On the fundamental group of one-dimensional spaces}, Topology Appl. 153 (2006) 2648-2672.

\bibitem{cannon2}
  J.W. Cannon, G.R. Conner, \emph{The combinatorial structure of the Hawaiian earring group}, Topol. Appl. 106 (2000) 225–271.

\bibitem{grcon}
  G.R. Conner, M. Meilstrup, D. Repovš, A. Zastrow, and M.Željko, \emph{On small
homotopies of loops}, Topology Appl. 155 (2008) 1089–1097.
  
\bibitem{dugundji}
J. Dugundji, \emph{A topologized fundamental group}, Proc. Nat. Acad. Sci. 36
(1950), 141–143.


\bibitem{eda}
  K. Eda, \emph{Free Subgroups of the Fundamental Group of the Hawaiian Earring}, J. Algebra. 219 (1999) 598-605.

\bibitem{engl}
  R. Engelking, \emph{General Topology}, Heldermann Verlag Berlin, 1989.
  
\bibitem{FZ07} H.~Fischer, A. ~Zastrow, \emph{Generalized universal covering spaces and the shape group}, Fund. Math . 197 (2007) 167 -- 196.

\bibitem{hurewicz}
W. Hurewicz, \emph{Homotopie, homologie und lokaler zusammenhang}, Fundamenta
Mathematicae 25 (1935), 467–485.

\bibitem{pakdam}
A. Pakdaman, H. Torabi, B. Mashayekhy, \emph{On the existence of categorical universal coverings}, Italian Journal of Pure and Applied Mathematics, \textbf{37} (2017) 289-300.

\bibitem{rashid}
M. Abdullahi Rashid, S.Z Pashaei, B. Mashayekhy, H. Torabi, \emph{On the Whisker Topology on Fundamental Group}. Confrence Paper from 46th Annual Iranian Mathematics Conference 46 (2015).

\bibitem{segal}
S. Mardešić and J. Segal, \emph{Shape theory}, North-Holland Publishing Com-
pany, 1982.

\bibitem{spanier}

E.H Spanier, \emph{Algebraic Topology}, McGraw-Hill New York, 1966.



\bibitem{zastrow}

Z. Virk, A. Zastrow, \emph{The comparison of topologies related to various concepts of generalized covering spaces}, Topology Appl. \textbf{170} (2014) 52-62.


\end{thebibliography}
\end{document}